\documentclass[a4paper,reqno,11pt]{amsart}
\usepackage{amscd}
\usepackage{amsmath}
\usepackage{amssymb}
\usepackage{amsthm}
\usepackage{latexsym}
\usepackage[all]{xypic}
\usepackage{amsfonts}
\usepackage{mathrsfs}
\usepackage{graphicx,color}
\usepackage{xcolor}
\usepackage{tikz}
\usepackage{graphics}
\usetikzlibrary[shapes]
\usetikzlibrary{cd,arrows,matrix,backgrounds,positioning,calc,decorations.markings,decorations.pathmorphing,decorations.pathreplacing}
\tikzset{black/.style={circle,fill=black,inner sep=3pt,outer sep=3pt},
	white/.style={circle,fill=white,draw=black,inner sep=3pt,outer sep=3pt},
}
\xyoption{all}
\usepackage[colorlinks=true,linkcolor=blue,citecolor=blue]{hyperref}
\usepackage{anysize,hyperref}
\marginsize{30mm}{30mm}{23mm}{23mm} % 设置左右页边距和上下页边距
\usepackage{array}
\newcolumntype{C}{>{$}c<{$}}

\usepackage{makecell}
\usepackage{adjustbox}
\usepackage{soul}
\setstcolor{red}

\def\red{\color{red}}

\usepackage{enumitem}
\newtheorem{theorem}{Theorem}[section]
\newtheorem{lemma}[theorem]{Lemma}
\newtheorem{corollary}[theorem]{Corollary}
\newtheorem{proposition}[theorem]{Proposition}
\theoremstyle{definition}
\newtheorem{definition}[theorem]{Definition}
\newtheorem{definition-proposition}[theorem]{Definition-Proposition}
\newtheorem{example}[theorem]{Example}

\def\C{\mathcal{C}}

\def\T{\mathcal{T}}

\def \text{\mbox}

\providecommand{\Ext}{\mathop{\rm Ext}\nolimits}%
\renewcommand{\mod}{\mathop{\rm mod}\nolimits}%
\providecommand{\ind}{\mathop{\rm ind}\nolimits}%

\usetikzlibrary{arrows,decorations.pathmorphing,decorations.pathreplacing,positioning,shapes.geometric,shapes.misc,decorations.markings,decorations.fractals,calc,patterns}
\begin{document}

\title{The Grothendieck groups of $n$-cluster categories of type $A_{\infty}$}

\author[Chang]{Huimin Chang}\thanks{$^\ast$Corresponding author.}
\address{Department of Applied Mathematics,
The Open University of China,
100039 Beijing,
P. R. China
}
\email{changhm@ouchn.edu.cn}

\begin{abstract}
In this article, we investigate the Grothendieck groups $K_0(\C_{A_{\infty}}^n)$ of $n$-cluster categories $\C_{A_{\infty}}^n$ of type $A_{\infty}$ introduced by T.~Holm and P.~J{\o}rgensen. We prove that $K_0(\C_{A_{\infty}}^n)\cong\mathbb{Z}$ for an arbitrary $n\geq 1$. As an application, this generalizes a result of Murphy for $n=1$. 
\end{abstract}

 %05E10; 18G80}


\thanks{Huimin Chang is supported by the National Natural Science Foundation of China (Grant No. 12301047).}

\maketitle

\section{Introduction}
In classical settings, the notion of a Grothendieck group has been used previously to give connections between
certain subcategories of a category and subgroups of the ambient category’s Grothendieck group; see \cite{M,T}. The Grothendieck group of a triangulated category $\C$  is defined to be the free abelian group of isomorphism classes of $\C$, modulo the Euler relations: $[B] = [A] + [C]$ for each triangle $A\rightarrow B\rightarrow C\rightarrow \Sigma A$ in $\C$. 

The Grothendieck group is an important numerical invariant of a category. Calculation of the Grothendieck group $K_0(\C)$ of a given category $\C$ has been extensively studied in recent years. One important method is connecting $K_0(\C)$ to a split Grothendieck group $K_0^{\rm sp}(\T)$ of a (higher) cluster tilting subcategory $\T$. We present some related results:
\begin{itemize}
  \item [(1)] When $\mathcal{C}$ is a locally finite triangulated category, Xiao and Zhu \cite{XZ} showed $K_0(\C)$ is isomorphic to $K_0(\C)$ modulo the Euler relations formed by Auslander-Reiten triangles. In this case, $\C$ can be seen the only 1-cluster tilting subcategory of $\C$.
  \item [(2)] When $\mathcal{C}$ is a $2$-Calabi-Yau triangulated category with a 2-cluster tilting subcategory $\T$, Palu \cite{P} proved that $K_0(\C)$ is isomorphic to $K_0^{\text{sp}}(\T)$ modulo the Euler relations formed by exchange triangles, and he calculated the Grothendieck groups of cluster categories of finite Dynkin type $A_n, D_n$ and $E_n$. Later Murphy \cite{Mu} calculated the Grothendieck groups of discrete cluster categories of Dynkin type $A_{\infty}$ by this method.
  \item [(3)]  When $\mathcal{C}$ is a  triangulated
 category with a Serre functor $\mathbb{S}$ and an $n$-cluster tilting subcategory $\T$, Fedele \cite{F} presented a higher cluster tilting version of Palu's result \cite{P}, and he computed the Grothendieck groups of $n$-cluster categories of finite type $A_m$ for odd $n$. Later, Chang and Zhuang \cite{CZ} calculated the Grothendieck groups of $n$-cluster categories of finite type $D_m$ for odd $n$ by this method. 
 \item [(4)]  When $\mathcal{C}$ is a Krull-Schmidt extriangulated
 category with an $n$-cluster titling subcategory $\T$, Wang, Wei and Zhang \cite{WWZ} proved that $K_0(\C)$ is isomorphic to the index Grothendieck group of $\T$ under some additional assumption. Recently, Wang \cite{Wa} generalized Wang, Wei and Zhang's result and  calculated the Grothendieck groups of $n$-cluster categories of finite type $A_m$, which covers Palu's \cite{P} and Fedele's \cite{F} results.
\end{itemize}

Let $n$ be an integer, $k$ an algebraically closed field, and $\C$ be a $k$-linear algebraic triangulated category that is idempotent complete and classically generated by an $(n+1)$-spherical object. This category was first studied in \cite{J} for the case $n = 1$. For different values of $n$, the category $\C$ appears in various well-known forms. When $n = -1$, $\C$ is equivalent to $D^c(k[X]/(X^2))$, the compact derived category of the dual numbers. When $n = 0$, it corresponds to $D^f(k[![X]!])$, the derived category of complexes with bounded finite-length homology over the formal power series ring. For $n = 1$, $\T$ is the cluster category of type $A_{\infty}$; see \cite{HJ}. When $n > 1$, it can be interpreted as the $n$-cluster category of type $A_{\infty}$; see \cite{HJ1}.

We note that $n$-cluster categories are a central and dynamic area of research in higher representation theory, as they provide a rich framework to explore the mutation and classification of $n$-cluster tilting subcategories. Moreover, they also serve as a bridge between combinatorial and categorical approaches to higher cluster algebras. 

In this paper, we investigate the Grothendieck groups $K_0(\C_{A_{\infty}}^n)$ of $n$-cluster categories $\C_{A_{\infty}}^n$ of type $A_{\infty}$ introduced by T.~Holm and P.~J{\o}rgensen. We prove that $K_0(\C_{A_{\infty}}^n)\cong\mathbb{Z}$ for an arbitrary $n\geq 1$. As an application, this generalizes a result of Murphy for $n=1$.

This paper is organized as follows. In Section~2, we recall basic notions of Grothendieck groups and related results. In Section~3, we review the construction of  $n$-cluster categories of type $A_{\infty}$, and a geometric interpretation of $(n+1)$-cluster tilting subcategories. In Section~4, we prove the main result.

\subsection*{Conventions}
Throughout this paper, we assume that $\mathcal{C}$ is a Hom-finite, Krull–Schmidt triangulated category with shift functor $\Sigma$.  The word subcategory will always mean full subcategory closed under isomor
phisms, directsums, and direct summands. In particular, a subcategory is determined by the indecomposable objects it contains.

\section{Grothendieck groups in triangulated categories}
This section provides background material on Grothendieck groups in triangulated categories and related results.

\subsection{Grothendieck group}
In this subsection, we recall the definition of Grothendieck group in a triangulated category $\C$.

$\mathbf{Notations}$: For an object $X$ in $\C$, we use $[X]$ to denote the isoclass of $X$ and $F$ to denote the free abelian group generated by the isoclasses $[X]$ of objects $X$ in $\C$.
\begin{definition}[\cite{F,XZ}]
By using the same notations as above, we define the split Grothendieck group of $\C$ to be
$$K_0^{\rm sp}(\C):=F/<[X\oplus Y]-[X]-[Y]>,$$
where $X,Y\in\C$. We define the Grothendieck group of $\C$ to be
$$K_0(\C):=K_0^{\rm sp}(\C)/<[X]-[Y]+[Z]\;|\;X\rightarrow Y\rightarrow Z\rightarrow \Sigma X \mathrm{\; is\;a\;triangle\;in\;}\C>.$$
\end{definition}
\subsection{Grothendieck groups with (higher) cluster tilting subcategories} In this subsection, we recall the relationship of Grothendieck groups of $\C$ with split Grothendieck groups of $\T$, where $\mathcal{T}$ is the (higher) cluster tilting subcategory of $\C$.
\begin{definition}
A subcategory $\T$ of $\C$ is called  $(n+1)$-cluster tilting if 
\begin{itemize}
  \item [(1)] $\T$ is functorially finite;
  \item [(2)] $\T=\{M\in\T|\Ext^i(M,\T)=0 \text{\;for\;} 1\leq i\leq n\}=\{M\in\T|\Ext^i(\T, M)=0 \text{\;for\;} 1\leq i\leq n\}$.
\end{itemize}
\end{definition}

Let $\mathcal{T}$ be an $(n+1)$-cluster tilting subcategory of $\C$. Wang proved the following for an extriangulated category; see \cite[Theorem 5.4]{Wa}.
\begin{lemma}\label{C} 
Let $\mathcal{T}$ be an $(n+1)$-cluster tilting subcategory in $\mathscr{C}$. Then $$K_{0}(\mathscr{C})\cong K_{0}^{\rm in}(\mathcal{T}).$$
	\end{lemma}
We explain the notation $K_{0}^{\rm in}(\mathcal{T})$ in the following.
$$K_{0}^{\rm in}(\mathcal{T})=K_{0}^{\rm sp}(\mathcal{T})/\langle\sum_{i=1}^{n+3}(-1)^{i+1}[X_{i}]~|~X_{1}{\longrightarrow}\cdots \stackrel{}{\longrightarrow} X_{n+2}\stackrel{}{\longrightarrow}X_{n+3}~$$
$$\text{is an $(n+3)$-angle with terms in  $\mathcal{T}$}\rangle.$$

The $(n+3)$-angle $$X_{1}\stackrel{a_1}{\longrightarrow}X_{2}\stackrel{a_2}{\longrightarrow}\cdots \stackrel{a_{n+1}}{\longrightarrow} X_{n+2}\stackrel{a_{n+2}}{\longrightarrow}X_{n+3}$$
means that there exists triangles 
$$C_i\stackrel{b_i}{\longrightarrow}X_{i+1}\stackrel{c_{i+1}}{\longrightarrow}C_{i+1}{\longrightarrow}\Sigma C_{i}$$
for $1\leq i\leq n+2$ such that $a_i=c_ib_i$, where $C_1=X_1, C_{n+2}=X_{n+3}, a_1=b_1$ and $c_{n+2}=a_{n+2}$. We sometimes denote it by
\begin{equation*}
				\begin{tikzcd}[column sep =10.5, row sep =10.5]
					&X_2 \ar[dr]  \ar[rr] && X_3 \ar[dr] \ar[r] & \cdots  &  \cdots &  \cdots\ar[r]&X_{n+1} \ar[dr]\ar[rr]  & & X_{n+2}    \ar[dr]&\\
					X_1\ar[ur] &&C_2 \ar[ll,rightarrow]\ar[ur]&   &C_3  \ar[ll,rightarrow] \ar[r]& \cdots\ar[r] &C_{n}    \ar[ur]& &C_{n+1}\ar[ur]\ar[ll,rightarrow]	&   &X_{n+3} .\ar[ll,rightarrow]
				\end{tikzcd}
			\end{equation*}
and omit the morphisms.
\section{$n$-cluster categories of type $A_{\infty}$}
In this section, we review the notion of $\C_{A_{\infty}}^n$, the geometric description of $\C_{A_{\infty}}^n$, and  $(n+1)$-cluster tilting subcategories of $\C_{A_{\infty}}^n$.

\subsection{Realization of $\C_{A_{\infty}}^n$}
\begin{definition}\cite[Definition 1.1]{HJ1}
We say that $\C$ is a $k$-linear algebraic triangulated category with suspension functor $\Sigma$, which is idempotent complete and classically generated by an $(n+1)$-spherical object $s$, if
\[
\dim_k\T(s,\Sigma^{\ell} s)=\left\{
\begin{array}{cl}
1  & \text{if\ $\ell = 0$ or $\ell = n+1$,} \\
0  & \text{otherwise.}
\end{array}
\right.
\]
\end{definition}
It was proved in \cite{KYZ,HJ1} that $\C$ is uniquely determined by these properties, and can be thought of as a $n$-cluster category of Dynkin type $A_{\infty}$.

From now on, suppose that $n\geq1$ is an integer and $\C_{A_{\infty}}^n$ is the $n$-cluster categories of type $A_{\infty}$. By \cite[Proposition 1.8]{HJY}, the Serre functor of $\C_{A_{\infty}}^n$ is $\mathbb{S}=\Sigma^{n+1}$, and by \cite[Proposition 1.10]{HJY}, the Auslander-Reiten quiver  of $\C_{A_{\infty}}^n$
consists of $n$ components, each of which is a copy of $\mathbb{Z}A_{\infty}$, and $\Sigma$ acts cyclically
on the set of components. The $n$ components are denoted by $R,\Sigma R,\cdots,\Sigma^{n-1} R$ and satisfies $\Sigma^{n+i} R=\Sigma^{i} R$ for $i=0,1,\cdots,n-1$.

We recall the geometric realization of $\C_{A_{\infty}}^n$ based on \cite{HJ1}.
\begin{definition}
\begin{itemize}
  \item[(1)] Let $\Pi$ be an $\infty$-gon. A pair of integers $(t, u)$ as an arc in $\Pi$
connecting the integers $t$ and $u$ with $u-t\geq2$ and $u-t\equiv 1(\mod n)$ is
called an $n$-admissible arc.
  \item[(2)] Let $\mathfrak{U}$ be a set of $n$-admissible arcs. An integer $t$ is a left-fountain of $\mathfrak{U}$ if $\mathfrak{U}$ contains infinitely many $n$-admissible arcs of the form $(s, t)$, and $t$ is a right-fountain of $\mathfrak{U}$ if $\mathfrak{U}$ contains infinitely many $n$-admissible arcs of the form $(t, u)$. We say that $t$ is a fountain of $\mathfrak{U}$ if it is both a left-fountain and a right-fountain of $\mathfrak{U}$.
   \item[(3)] We say that $\mathfrak{U}$ is locally finite if, for each integer $t$, there are only finitely many $n$-admissible arcs of the form $(s, t)$ and $(t, u)$ in $\mathfrak{U}$.    
\end{itemize}
\end{definition}
The suspension functor $\Sigma$, the Serre funcotor $\mathbb{S}$, and the AR translation $\tau=\mathbb{S}\Sigma^{-1}$ of $\C_{A_{\infty}}^n$ act on an $n$-admissible arc $(t, u)$ are given by
$$\Sigma(t, u)=(t-1,u-1),\;\;\;\mathbb{S}(t, u)=(t-n-1,u-n-1),\;\;\;\tau(t, u)=(t-n,u-n).$$
Note that $\tau=\Sigma^n$.
\begin{proposition}\cite[Proposition 2.4]{HJ1}
There is a bijective correspondence between $\ind \C_{A_{\infty}}^n$ and the set of $n$-admissible arcs. Moreover, This extends to a bijective correspondence between subcategories of $\C_{A_{\infty}}^n$ and subsets of the set of $n$-admissible arcs.
\end{proposition}
In the following, we sometimes use an $n$-admissible arc $(t, u)$ to represent an indecomposable object $M$ in $\C_{A_{\infty}}^n$ and write $M = (t, u)$ without confusion.
\begin{example}
Let $n=3$. The Auslander-Reiten quiver  of $\C_{A_{\infty}}^3$ contains 3 components $R,\Sigma R$ and $\Sigma^2R$. We draw the  Auslander-Reiten quiver  of  $\C_{A_{\infty}}^3$, see Figures \ref{figure1}-\ref{figure3}.
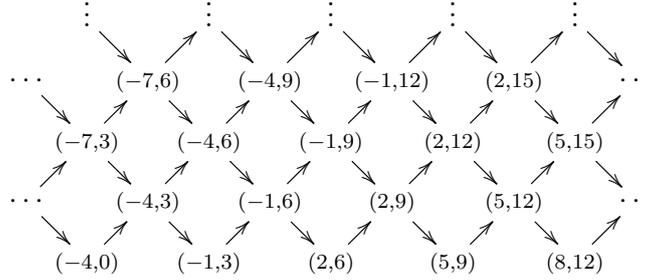
\begin{figure}[h]
\[
  \xymatrix @-3.0pc @! {
    & \vdots \ar[dr] & & \vdots \ar[dr] & & \vdots \ar[dr] & & \vdots \ar[dr] & & \vdots \ar[dr] & \\
    \cdots \ar[dr]& & {\scriptstyle (-7,6)} \ar[ur] \ar[dr] & & {\scriptstyle (-4,9)} \ar[ur] \ar[dr] & & {\scriptstyle (-1,12)} \ar[ur] \ar[dr] & & {\scriptstyle (2,15)} \ar[ur] \ar[dr] & & \cdots \\
    & {\scriptstyle (-7,3)} \ar[ur] \ar[dr] & & {\scriptstyle (-4,6)} \ar[ur] \ar[dr] & & {\scriptstyle (-1,9)} \ar[ur] \ar[dr] & & {\scriptstyle (2,12)} \ar[ur] \ar[dr] & & {\scriptstyle (5,15)} \ar[ur] \ar[dr] & \\
    \cdots \ar[ur]\ar[dr]& & {\scriptstyle (-4,3)} \ar[ur] \ar[dr] & & {\scriptstyle (-1,6)} \ar[ur] \ar[dr] & & {\scriptstyle (2,9)} \ar[ur] \ar[dr] & & {\scriptstyle (5,12)} \ar[ur] \ar[dr] & & \cdots\\
    & {\scriptstyle (-4,0)} \ar[ur] & & {\scriptstyle (-1,3)} \ar[ur] & & {\scriptstyle (2,6)} \ar[ur] & & {\scriptstyle (5,9)} \ar[ur] & & {\scriptstyle (8,12)} \ar[ur] & \\
               }
\]
\caption{The component $R$ of the Auslander-Reiten quiver  of $\C_{A_{\infty}}^3$.}
\label{figure1}
\end{figure}
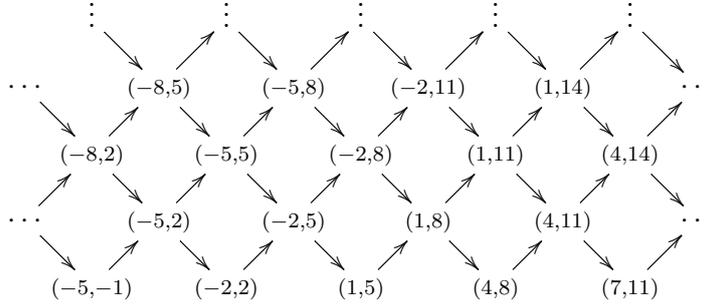
\begin{figure}[h]
\[
  \xymatrix @-3.0pc @! {
    & \vdots \ar[dr] & & \vdots \ar[dr] & & \vdots \ar[dr] & & \vdots \ar[dr] & & \vdots \ar[dr] & \\
    \cdots \ar[dr]& & {\scriptstyle (-8,5)} \ar[ur] \ar[dr] & & {\scriptstyle (-5,8)} \ar[ur] \ar[dr] & & {\scriptstyle (-2,11)} \ar[ur] \ar[dr] & & {\scriptstyle (1,14)} \ar[ur] \ar[dr] & & \cdots \\
    & {\scriptstyle (-8,2)} \ar[ur] \ar[dr] & & {\scriptstyle (-5,5)} \ar[ur] \ar[dr] & & {\scriptstyle (-2,8)} \ar[ur] \ar[dr] & & {\scriptstyle (1,11)} \ar[ur] \ar[dr] & & {\scriptstyle (4,14)} \ar[ur] \ar[dr] & \\
    \cdots \ar[ur]\ar[dr]& & {\scriptstyle (-5,2)} \ar[ur] \ar[dr] & & {\scriptstyle (-2,5)} \ar[ur] \ar[dr] & & {\scriptstyle (1,8)} \ar[ur] \ar[dr] & & {\scriptstyle (4,11)} \ar[ur] \ar[dr] & & \cdots\\
    & {\scriptstyle (-5,-1)} \ar[ur] & & {\scriptstyle (-2,2)} \ar[ur] & & {\scriptstyle (1,5)} \ar[ur] & & {\scriptstyle (4,8)} \ar[ur] & & {\scriptstyle (7,11)} \ar[ur] & \\
               }
\]

\caption{The component $\Sigma R$ of the Auslander-Reiten quiver  of $\C_{A_{\infty}}^3$.}
\label{figure2}
\end{figure}
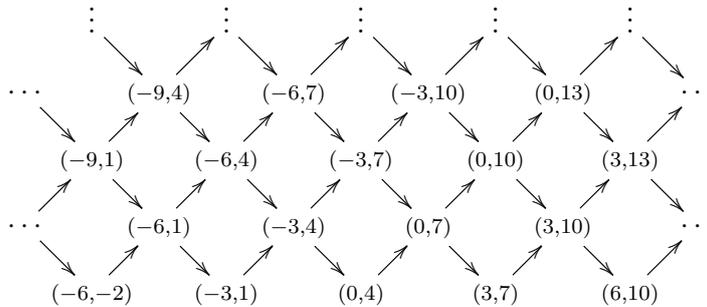
\begin{figure}[h]
\[
  \xymatrix @-3.0pc @! {
    & \vdots \ar[dr] & & \vdots \ar[dr] & & \vdots \ar[dr] & & \vdots \ar[dr] & & \vdots \ar[dr] & \\
    \cdots \ar[dr]& & {\scriptstyle (-9,4)} \ar[ur] \ar[dr] & & {\scriptstyle (-6,7)} \ar[ur] \ar[dr] & & {\scriptstyle (-3,10)} \ar[ur] \ar[dr] & & {\scriptstyle (0,13)} \ar[ur] \ar[dr] & & \cdots \\
    & {\scriptstyle (-9,1)} \ar[ur] \ar[dr] & & {\scriptstyle (-6,4)} \ar[ur] \ar[dr] & & {\scriptstyle (-3,7)} \ar[ur] \ar[dr] & & {\scriptstyle (0,10)} \ar[ur] \ar[dr] & & {\scriptstyle (3,13)} \ar[ur] \ar[dr] & \\
    \cdots \ar[ur]\ar[dr]& & {\scriptstyle (-6,1)} \ar[ur] \ar[dr] & & {\scriptstyle (-3,4)} \ar[ur] \ar[dr] & & {\scriptstyle (0,7)} \ar[ur] \ar[dr] & & {\scriptstyle (3,10)} \ar[ur] \ar[dr] & & \cdots\\
    & {\scriptstyle (-6,-2)} \ar[ur] & & {\scriptstyle (-3,1)} \ar[ur] & & {\scriptstyle (0,4)} \ar[ur] & & {\scriptstyle (3,7)} \ar[ur] & & {\scriptstyle (6,10)} \ar[ur] & \\
               }
\]

\caption{The component $\Sigma^2 R$ of the Auslander-Reiten quiver  of $\C_{A_{\infty}}^3$.}
\label{figure3}
\end{figure}
\end{example}
\subsection{$(n+1)$-cluster tilting subcategories}
In this subsection, we recall the geometric description of $(n+1)$-cluster tilting subcategories of $\C_{A_{\infty}}^n$.

\begin{definition}
We say that $\mathfrak{U}$ is an $(n+2)$-angulation of the ${\infty}$-gon if it is a maximal set of pairwise non-crossing $n$-admissible arcs.
\end{definition}

\begin{lemma}\cite[Theorem E]{HJ1}
\label{aa}
The subcategory $\T$ of $\C_{A_{\infty}}^n$ is $(n+1)$-cluster tilting if and only if the corresponding set of $n$-admissible arcs is an $(n+2)$-angulation of the $\infty$-gon which is either locally finite or has a fountain.
\end{lemma}

\begin{example}\label{excluster}
Let $$T_{2k}=(1-kn,2+kn),\; k=1,2,\cdots$$
 $$T_{2k+1}=(1-kn,2+(k+1)n), \;k=0,1,2,\cdots$$
It is easy to check that the set of $n$-admissible arcs $\{T_1,T_2,\cdots\}$ is an $(n+2)$-angulation of the ${\infty}$-gon and locally finite, see Figure \ref{n-cluster}. Thus, it corresponds to an $(n+1)$-cluster tilting subcategory of $\C_{A_{\infty}}^n$. We denote this subcategory by $\T$.
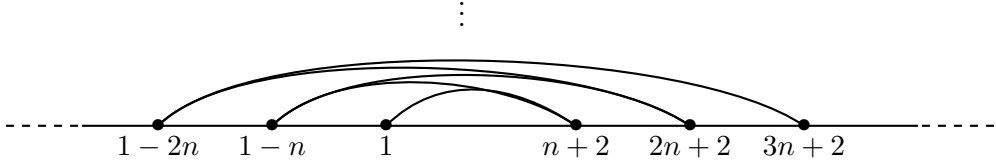
\begin{figure}[ht]\centering
\begin{tikzpicture}
\draw[thick] (-8,0) to (3,0);
\draw[thick, dashed] (-9,0) to (-8,0) (4,0) to (3,0);
\draw[thick] (-7,0)node{$\bullet$}node[below]{$1-2n$}(-5.5,0)node{$\bullet$}node[below]{$1-n$}(-4,0)node{$\bullet$}node[below]{1}
(-1.5,0)node{$\bullet$}node[below]{$n+2$} (0,0)node{$\bullet$}node[below]{$2n+2$} (1.5,0)node{$\bullet$}node[below]{$3n+2$}
(-3,2)node[below]{$\vdots$};
\draw[thick] (-5.5,0) ..  controls +(45:1.2) and +(145:1.2) .. (-1.5,0);
\draw[thick] (-4,0) ..  controls +(45:1) and +(145:1) .. (-1.5,0);
\draw[thick] (-5.5,0) ..  controls +(45:1.4) and +(145:1.4) .. (0,0);
\draw[thick] (-7,0) ..  controls +(45:1.6) and +(145:1.6) .. (0,0);
\draw[thick] (-7,0) ..  controls +(45:1.8) and +(145:1.8) .. (1.5,0);
\end{tikzpicture}
\caption{An $(n+2)$-angulation of the ${\infty}$-gon and locally finite}
\label{n-cluster}
\end{figure}
\end{example}

\section{Grothendieck groups of $\C_{A_{\infty}}^n$}
The main result of this paper is the following.
\begin{theorem}
Let $K_0(\C_{A_{\infty}}^n)$ be the Grothendieck groups of $\C_{A_{\infty}}^n$. Then we have
$$K_0(\C_{A_{\infty}}^n)=\mathbb{Z}.$$
\end{theorem}
\begin{proof}
Take $$T_{2k}=(1-kn,2+kn),\; k=1,2,\cdots$$
 $$T_{2k+1}=(1-kn,2+(k+1)n), \;k=0,1,2,\cdots$$
as in Example \ref{excluster}. Then $\T=\bigoplus_{i=1}^{\infty}T_i$ is an $(n+1)$-cluster tilting subcategory of $\C_{A_{\infty}}^n$. 
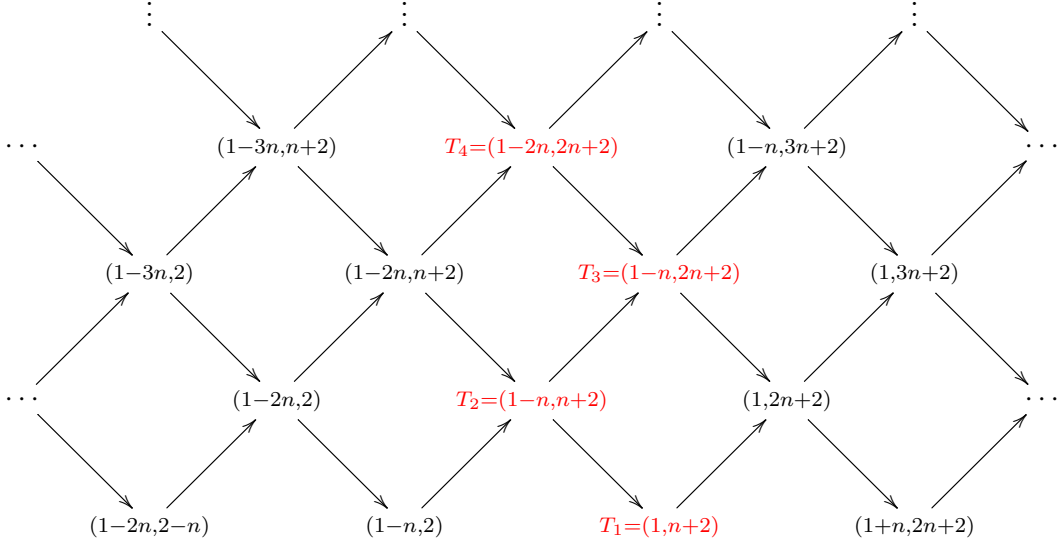
\begin{figure}[h]
\[
  \xymatrix @-4.0pc @! {
    & \vdots \ar[dr] & & \vdots \ar[dr] & & \vdots \ar[dr] & & \vdots \ar[dr] &  \\
    \cdots \ar[dr]& & {\scriptstyle (1-3n,n+2)} \ar[ur] \ar[dr] & & {\scriptstyle {\red T_4=(1-2n,2n+2)}} \ar[ur] \ar[dr] & & {\scriptstyle (1-n,3n+2)} \ar[ur] \ar[dr]  & & \cdots \\
    & {\scriptstyle (1-3n,2)} \ar[ur] \ar[dr] & & {\scriptstyle (1-2n,n+2)} \ar[ur] \ar[dr] & & {\scriptstyle {\red T_3=(1-n,2n+2)}} \ar[ur] \ar[dr] & & {\scriptstyle (1,3n+2)} \ar[ur] \ar[dr] & \\
    \cdots \ar[ur]\ar[dr]& & {\scriptstyle (1-2n,2)} \ar[ur] \ar[dr] & & {\scriptstyle {\red T_2=(1-n,n+2)}} \ar[ur] \ar[dr] & & {\scriptstyle (1,2n+2)} \ar[ur] \ar[dr] & & \cdots\\
    & {\scriptstyle (1-2n,2-n)} \ar[ur] & & {\scriptstyle (1-n,2)} \ar[ur] & & {\scriptstyle {\red T_1=(1,n+2)}} \ar[ur] & & {\scriptstyle (1+n,2n+2)} \ar[ur] & \\
               }
\]

\caption{The component containing $T_1$ of the Auslander-Reiten quiver  of $\C_{A_{\infty}}^n$.}
\label{figure8}
\end{figure}

First, we consider the $(n+3)$-angle starting at $T_1$ and ending at $T_1$. Note that there is an Auslander-Reiten triangle
$$(1-n,2)\rightarrow T_{2}\rightarrow T_{1}$$
by Figure \ref{figure8}, and $\Sigma^{n}(1,n+2)=(1-n,2)$, then there exists an $(n+3)$-angle starting at $T_1$ and ending at $T_1$:
\begin{equation*}
				\begin{tikzcd}[column sep =10.5, row sep =10.5]
					&0 \ar[dr]  \ar[rr]&& 0 \ar[dr] \ar[r] & \cdots  &  \cdots &  \cdots\ar[r]&0 \ar[dr]\ar[rr]  & & T_2   \ar[dr]&\\
					T_1\ar[ur] &&\Sigma T_1 \ar[ll,rightarrow]\ar[ur]&   &\Sigma^{2} T_1  \ar[ll,rightarrow] \ar[r]& \cdots\ar[r] &\Sigma^{n-1}T_1   \ar[ur]& &\Sigma^{n}T_1\ar[ur]\ar[ll,rightarrow]	&   &T_1.\ar[ll,rightarrow]
				\end{tikzcd}
			\end{equation*}
Thus, we have $[T_1]+(-1)^{n+1}[T_2]+(-1)^{n+2}[T_1]=0$ in  $K_{0}^{\rm in}(\T)$.

For different values of $i$, the $(n+3)$-angle starting at $T_i$ and ending at $T_i$ is different, we have the following two cases:
\begin{itemize}
  \item [(1)] For the object $T_{2k+1}$ with $k=1,2,\cdots$, there exists an Auslander-Reiten triangle
$$\Sigma^nT_{2k+1}\rightarrow T_{2k}\oplus T_{2k+2}\rightarrow T_{2k+1}$$
by Figure \ref{figure8}. This implies there exists an $(n+3)$-angle starting at $T_{2k+1}$ and ending at $T_{2k+1}$:
\begin{equation*}
				\begin{tikzcd}[column sep =9, row sep =9]
					&0 \ar[dr]  \ar[rr]&& 0 \ar[dr] \ar[r] & \cdots  &  \cdots &  \cdots\ar[r]&0 \ar[dr]\ar[rr]  & & T_{2k}\oplus T_{2k+2}   \ar[dr]&\\
					T_{2k+1}\ar[ur] &&\Sigma T_{2k+1} \ar[ll,rightarrow]\ar[ur]&   &\Sigma^{2} T_{2k+1}   \ar[ll,rightarrow] \ar[r]& \cdots\ar[r] &\Sigma^{n-1}T_{2k+1}    \ar[ur]& &\Sigma^{n}T_{2k+1}\ar[ur]\ar[ll,rightarrow]	&   &T_{2k+1}.\ar[ll,rightarrow]
				\end{tikzcd}
			\end{equation*}
Thus, we have $[T_{2k+1}]+(-1)^{n+1}[T_{2k}\oplus T_{2k+2}]+(-1)^{n+2}[T_{2k+1}]=0$ in  $K_{0}^{\rm in}(\T)$.
  \item [(2)]For the object $T_{2k}$ with $k=1,2,\cdots$, since $\Sigma^{-n}T_{2k}=(1-(k-1)n,2+(k+1)n)$, and there exists an Auslander-Reiten triangle
$$T_{2k}\rightarrow T_{2k-1}\oplus T_{2k+1}\rightarrow \Sigma^{-n}T_{2k}$$
by Figure \ref{figure8},  there exists an $(n+3)$-angle starting at $T_{2k}$ and ending at $T_{2k}$:
\begin{equation*}
				\begin{tikzcd}[column sep =9, row sep =9]
					& T_{2k-1}\oplus T_{2k+1} \ar[dr]  \ar[rr]&& 0 \ar[dr] \ar[r] & \cdots  &  \cdots &  \cdots\ar[r]&0 \ar[dr]\ar[rr]  & & 0   \ar[dr]&\\
					T_{2k}\ar[ur] &&\Sigma^{-n}T_{2k} \ar[ll,rightarrow]\ar[ur]&   &\Sigma^{-n+1}T_{2k}   \ar[ll,rightarrow] \ar[r]& \cdots\ar[r] &\Sigma^{-2}T_{2k}    \ar[ur]& &\Sigma^{-1}T_{2k}\ar[ur]\ar[ll,rightarrow]	&   &T_{2k}.\ar[ll,rightarrow]
				\end{tikzcd}
			\end{equation*}
Then we have $[T_{2k}]-[T_{2k-1}\oplus T_{2k+1}]+(-1)^{n+2}[T_{2k}]=0$ in  $K_{0}^{\rm in}(\T)$.
\end{itemize}

Combing the above cases, we have that:
\begin{itemize}
  \item [(1)] Suppose $n$ is odd. Then 
  $$[T_2]=0, [T_{2k+2}]+[T_{2k}]=0, [T_{2k+1}]+[T_{2k-1}]=0,$$
  for $k=1,2,\cdots$. This implies 
  $$[T_{2k}]=0, [T_{2k+1}]=(-1)^k[T_1]$$
  for $k=1,2,\cdots$. Hence $$K_0(\C_{A_{\infty}}^n)=\mathbb{Z}.$$
  \item [(2)] Suppose $n$ is even. Then 
  $$[T_2]=2[T_1], 2[T_{2k+1}]=[T_{2k+2}]+[T_{2k}], 2[T_{2k}]=[T_{2k+1}]+[T_{2k-1}],$$
  for $k=1,2,\cdots$. 
  We show $$[T_i]=i[T_1]$$ with $i=1,2,\cdots$ by mathematical induction in the following. Since $[T_3]=2[T_2]-[T_1]=3[T_1]$, this implies $[T_i]=i[T_1]$ holds for $i=3$. Suppose the equation holds for $i\leq a$, we show it holds for $a+1$.
  
  If $a+1$ is odd, then $a+1=2k+1$ for some $k$. Thus
  $$[T_{a+1}]=[T_{2k+1}]=2[T_{2k}]-[T_{2k-1}]=4k[T_1]-(2k-1)[T_1]=(a+1)[T_1].$$
  
  If $a+1$ is even, then $a+1=2k+2$ for some $k$. Thus
  $$[T_{a+1}]=[T_{2k+2}]=2[T_{2k+1}]-[T_{2k}]=2(2k+1)[T_1]-2k[T_1]=(a+1)[T_1]$$  
  This shows $$[T_i]=i[T_1]$$ with $i=1,2,\cdots$. Hence $$K_0(\C_{A_{\infty}}^n)=\mathbb{Z}.$$
\end{itemize}
\end{proof}

We find an interesting phenomenon: the Grothendieck groups of $\C_{A_{\infty}}^n$ do not depend on $n$. Note that when $n=1$, $\C_{A_{\infty}}^1$ is the cluster category of type $A_{\infty}$, so we have the following corollary.
\begin{corollary}\cite[Theorem 4.6]{Mu}
The Grothendieck group of $\C_{A_{\infty}}^1$ is
$$K_0(\C_{A_{\infty}}^1)=\mathbb{Z}.$$
\end{corollary}

\vspace{1cm}
\hspace{-5mm}\textbf{Data Availability}\hspace{2mm} Data sharing not applicable to this article as no datasets were generated or analysed during
the current study.
\vspace{2mm}

\hspace{-5mm}\textbf{Conflict of Interests}\hspace{2mm} The author declare that she has no conflicts of interest to this work.

\end{document}